\documentclass[draft,11pt,leqno]{article}
\usepackage{amsmath,amssymb,amsthm}
\usepackage{amsthm,amssymb,amsmath,amsfonts} 
\usepackage{eucal}      
\usepackage{mathrsfs}   
\usepackage{longtable}
\usepackage{multirow}
\usepackage{verbatim}
\usepackage{placeins}
\usepackage{graphicx}
\usepackage[T1]{fontenc}
\usepackage{color}

\usepackage{lineno,hyperref}

\def\Rset{\mathbb{R}}

\theoremstyle{plain}
\newtheorem{thm}{Theorem}[section]

\newtheorem{cor}[thm]{Corollary}
\newtheorem{prop}[thm]{Proposition}

\theoremstyle{definition}
\newtheorem{defn}[thm]{Definition}
\newtheorem{rem}[thm]{Remark}
\newtheorem{exmp}[thm]{Example}

\title{\bf Why and How the Definition of the Conformable Derivative in the Lower Terminal Should be Changed}

\author{Hristo Kiskinov $^{1}$, Milena Petkova $^{2}$, Andrey Zahariev $^{3}$\vspace{0.2cm}\\ 
Faculty of Mathematics and Informatics, University of Plovdiv, \\Plovdiv 4000, Bulgaria\vspace{0.2cm}\\
$^{1}$  kiskinov@uni-plovdiv.bg;
$^{2}$  milenapetkova@uni-plovdiv.bg;\\
$^{3}$  zandrey@uni-plovdiv.bg
}
\date{}

\begin{document}
\maketitle

\begin{abstract}

This paper discusses some unusual consequences raised by the definition of the conformable derivative 
in the lower terminal. 
A replacement for this definition is proposed and statements adjusted to the new definition are presented.
	

\end{abstract}

\section{Introduction}
\label{sec:introduction}

The fractional  calculus (now almost 330 year old) has attracted many researches in the last and present centuries.
There are a lot of definitions of fractional derivatives with different properties.
For a good introduction on the fractional calculus theory and fractional differential equations 
with the "classical" nonlocal fractional derivatives of Riemann-Liouville and Caputo,
see the monographs of Kilbas et al. \cite{KST06}, Samko et al. \cite{SKM93} and Podlubny \cite{Pod99}. 

In 2014, Khalil, Al Horani, Yousef and Sababheh \cite{KHYS14} 
introduced a definition of a local kind derivative called from the authors conformable fractional derivative. 
As an important reason for its introduction is specified the fact 
that this derivative satisfies a big part from the well-known properties of the integer order derivatives. 

In 2015, Abdeljawad \cite{Abd15} made an extensive research of the newly introduced conformable fractional calculus. 
In \cite{Mar18} Martynyuk (there  and in   \cite{MS18,MSS19-1,MSS19-2} 
the name "fractional-like" instead "conformable fractional" derivative is used) presented a physical interpretation of this derivative.
Taking in account \cite{AU15, OM15, OM17, Giu18, OM18, OMFM19, Tar18}, since the derivative is local, 
we will use the name "conformable derivative" instead as introduced "conformable fractional derivative".
In \cite{KPZ19} is studied the relationship between the conformable derivatives of different order 
and is obtained that a function has a conformable derivative at a point if and only if it has a first order derivative at the same point and that holds for all points except the lower terminal. The same result is presented also in \cite{A18, A19}.
The conformable derivative in arbitrary Banach space is introduced in \cite{KPZ21}.

Although they does not fulfill all the requirements to be classified as classical fractional derivatives, 
conformable derivatives turn out to be very convenient to work with. 
For this reason, there has been a significant increase in publications (more than hundred research articles) 
concerning or using conformal derivatives in the recent years.
See for example 
\cite{AM23}--\cite{AAJ23}, 
\cite{AER23}--\cite{FRE23},
\cite{HYB24}--\cite{K24},
\cite{LSGRER24}--\cite{LKAA23},
\cite{MWR19}--\cite{MFIASC24},
\cite{RIZGR23},
\cite{SAJA23}--\cite{SRAI23},
\cite{TNOCN23}--\cite{ZFW15}  
and the references therein.

In this paper we point out and discuss some unusual consequences 
raised by the definition of the conformable derivative in the lower terminal. 
In the existing definition the conformable derivative at the lower terminal is introduced as continuous prolongation 
of the conformable derivative in the inner points and does not depend from the value of the function at the lower terminal.
This is a big difference in comparison with the definition of the first order derivative and brings a lot of tricky behaviours. 
This must be taken into account by many statements using conformable derivatives and that is why some of them are formulated not in the usual way. 
It may be noted that a little carelessness connected with not taking these unusual features into account leads to inaccuracies such as unfortunately exist in some works.
A replacement for this definition is proposed and new rewritten statements adjusted to the new definition are presented.

We are convinced that the corrected definition of a conformable derivative proposed in this article 
is more convenient and natural than the one currently existing and we recommend to be used in the future.

The paper is organized as follows: 
In Section~\ref{sec:2}
we give definitions and some important properties for the conformable derivatives. 
In Section~\ref{sec:3} we list six problems caused by the definition of the conformable derivative in  the lower terminal,
which must be avoided.
A replacement of the discussed definition is offered in Section~\ref{sec:4}.
In Section~\ref{sec:5} we list how some of the statements, presented in the Section Preliminaries,
will change with the replacement of the definition of the conformable derivative. 
Section~\ref{sec:Conclusions} is devoted to our comments of the obtained results.

%
%

\section{Preliminaries} 
\label{sec:2}

For convenience and to avoid possible misunderstandings, below we recall the definitions of the 
conformable integral
and the conformable derivative (as introduced in  \cite{KHYS14}) as well as some important properties. 

\begin{defn} [\cite{Abd15}, \cite{KHYS14}]   \label{d2.1}
The left-sided conformable derivative  of order ${\alpha \in (0,1]}$ at the point 
${t \in (a,\infty)}$ for a function ${f \colon [a,\infty) \to \Rset}$ is defined by 
\begin{equation}   \label{eq1}
	T^{\alpha}_a f(t) = \lim_{\theta \to 0}\left(\frac{f(t+\theta(t-a)^{1-\alpha}) - f(t)}{\theta}\right)
\end{equation}
if the limit exists. 
\end{defn}

As in the case of the classical fractional derivatives the point ${a \in \Rset}$ appearing in \eqref{eq1} 
will  be called lower terminal of the left-sided conformable derivative. 
If for $f$ the conformable derivative of order $\alpha$ exists, then we will say that $f$ is $\alpha$-differentiable.

\begin{defn} [\cite{Abd15}, \cite{KHYS14}]   \label{d2.2}
The ${\alpha}$-derivative of $f$ at the lower terminal point $a$ 
in the case when $f$ is  ${\alpha}$-differentiable 	in some interval ${(a,a+\varepsilon)}$, ${\varepsilon > 0}$ 
is defined as 
$$
	T^{\alpha}_a f(a) = \lim_{t \to a+} T^{\alpha}_a f(t)
$$
if the limit exists.
\end{defn} 

If $f$ is $\alpha$-differentiable in some finite or infinite interval 
${J \subset [a,\infty)}$ 
we will write that  ${f \in C^{\alpha}_a (J, \Rset)}$, 
where with the indexes $a$ and $\alpha$ are denoted the lower terminal and the order of the conformable derivative respectively.

It may be noted, that some authors (see for example \cite{Abd15}) 
use the notation $T_{\alpha}^a $  instead $T^{\alpha}_a $, 
but we prefer to follow the traditions from the notations of the classical fractional derivatives 
and will write the lower terminal below and the order above. 

\begin{defn} [\cite{Abd15}, \cite{KHYS14}] \label{d1.3}
For each ${t > a}$ 
the left-sided conformable integral of order ${\alpha \in (0,1]}$ with lower terminal ${a \in \Rset}$ 
is defined by
\begin{equation}   \label{eq2}
I^{\alpha}_a f(t) = \int_a^t (s - a)^{\alpha - 1}\,f(s) {\rm d}s
\end{equation}
where the integral is the usual Riemann improper integral.
\end{defn}

Since in our exposition below we will use only left-sided conformable derivatives and integrals, 
then for shortness we will omit the expression "left-sided". 


\begin{thm} [\cite{Abd15}, \cite{KHYS14}]   \label{t2.4}
Let ${\alpha \in (0,1]}$.
If $f:[a,\infty) \to \Rset$  is $\alpha$-differentiable at $t_0 \in (a,\infty)$,
then $f$ is continuous at $t_0$. 
\end{thm}

\begin{thm} [\cite{Abd15}, \cite{KHYS14}]   \label{t2.5}
	Let ${\alpha \in (0,1]}$, $c,d \in \Rset$ and $J \subset (a,\infty)$. We assume that ${f,g \in C^{\alpha}_a (J, \Rset)}$. 
	Then for ${t \in J}$ the following relations hold:
	\begin{enumerate}
		\item [$(i)$] ${T^{\alpha}_a (cf+dg) = c\,T^{\alpha}_a f + d\,T^{\alpha}_a g}$;\vspace{0.1cm}
		\item [$(ii)$] ${T^{\alpha}_a (fg) = g\,T^{\alpha}_a f + f\,T^{\alpha}_a g}$;\vspace{0.1cm}
		\item [$(iii)$] ${T^{\alpha}_a (fg^{-1}) = \left(f\,T^{\alpha}_a g - g\,T^{\alpha}_a f\right)g^{-2}}$;\vspace{0.1cm}
		\item [$(iv)$] ${T^{\alpha}_a (1) = 0}$;\vspace{0.1cm}
		\item [$(v)$] ${T^{\alpha}_a f(t) = (t - a)^{1-\alpha}f'(t)}$ if in addition $f$ is differentiable for ${t \in J}$.
	\end{enumerate}
\end{thm}

\begin{thm}   [\cite{KPZ19}]  \label{t2.6}
	Let ${f : [a,\infty) \to \Rset}$ and there exist a point ${t_0 \in (a,\infty)}$ and number ${\alpha \in (0, 1]}$ 
	such that the conformable derivative ${T^{\alpha}_a f(t_0)}$ with lower terminal point $a$ exists.
	
	Then the conformable derivative ${T^{\beta}_a f(t_0)}$ exists 
	for every ${\beta \in (0, 1]}$ with ${\beta \neq \alpha}$ and 
$$
	{T^{\alpha}_a f(t_0) = (t_0 - a)^{\beta-\alpha}T^{\beta}_a f(t_0)}.
$$
\end{thm}

\begin{cor}    [\cite{KPZ19}]  \label{c2.7}
	For a function ${f : [a,\infty) \to \Rset}$ the conformable derivative ${T^{\alpha}_a f(t_0)}$ 
	with lower terminal $a$ at a point ${t_0 \in (a,\infty)}$ for some ${\alpha \in (0,1)}$ exists 
	if and only if 
	the function ${f(t)}$ has first derivative at the point ${t_0 \in (a,\infty)}$ and
\begin{equation} \label{e1}
	{T^{\alpha}_a f(t_0) = (t_0 - a)^{1-\alpha}  f'(t_0)}.
\end{equation}
\end{cor}

\begin{rem}
Corollary \ref{c2.7} shows the close connection between the conformable derivative and the first order derivative on 
$(a,\infty)$. It shows also, that the additional condition $f$ to be differentiable 
in Pos. {\it (v)} in Theorem \ref{t2.5} is not necessary.  
This result is obtained also in \cite{A18} in the particular case for conformable derivatives with lower terminal zero. 
See also \cite{A19}. 
\end{rem}

The next propositions treat the problem of the left and right inverse operator of the conformable derivative.
Because we did not have found the same statements as Proposition \ref{p2.11} and Proposition \ref{p2.9} in the literature, 
we also present their proofs. 

\begin{prop}   [\cite{Abd15}, \cite{KHYS14}]  \label{p2.8}
	Let ${f : [a,\infty) \to \Rset)}$ be continuous on $[a,\infty)$. 
	
	Then ${T^{\alpha}_a I^{\alpha}_a f(t_0) = f(t_0)}$ for ${t_0 \in (a,\infty)}$.
\end{prop}

\begin{prop}    \label{p2.11}
	Let for  the function ${f : [a,\infty) \to \Rset}$ for some ${\alpha \in (0,1)}$ be fulfilled
 ${f \in C_a^{\alpha}((a,\infty),\Rset)}$.
	
	Then for every ${t_0 \in (a,\infty)}$ we have that 
\begin{equation} \label {eq4}
{I^{\alpha}_a\,T^{\alpha}_a f(t_0) = f(t_0) - \lim\limits_{t \to a+} f(t)}.
\end{equation}
 if the limit exists and is finite. 
\end{prop}
\begin{proof}
Applying \eqref{eq2} and \eqref{e1} we obtain 
\begin{equation} \nonumber
I^{\alpha}_a\,T^{\alpha}_a f(t_0) 
=  \int_a^{t_0} (s - a)^{\alpha - 1}   (s - a)^{1-\alpha}  f'(s)  {\rm d}s
\end{equation}
which is an improper integral of second kind because 
$(s - a)^{\alpha - 1}   (s - a)^{1-\alpha}  f'(s)$ 
is not defined for $s=a$.
Then if  $\lim\limits_{t \to a+} f(t)$ exists and is finite we obtain for ${t_0 \in (a,\infty)}$
\begin{equation} \nonumber
\begin{split}
I^{\alpha}_a\,T^{\alpha}_a f(t_0) 
& =  \int_a^{t_0} (s - a)^{\alpha - 1}   (s - a)^{1-\alpha}  f'(s)  {\rm d}s \\
& =\lim_{\varepsilon \to 0+} \int_{a+\varepsilon}^{t_0}   (s - a)^{\alpha - 1}   (s - a)^{1-\alpha}  f'(s)  {\rm d}s \\
&= \lim_{\varepsilon \to 0+} \int_{a+\varepsilon}^{t_0}   f'(s)  {\rm d}s 
=\lim_{\varepsilon \to 0+} (f(t_0)-f(a+\varepsilon)) \\
& =f(t_0)-\lim_{\varepsilon \to 0+} f(a+\varepsilon)
\end{split}
\end{equation}
 \end{proof}

\begin{prop}    \label{p2.9}
	Let the following conditions hold:
	\begin{enumerate}
		\item [$1.$] The function ${f : [a,\infty) \to \Rset}$ is right continuous at $a$.
		\item [$2.$] For some ${\alpha \in (0,1)}$ the function ${f \in C_a^{\alpha}((a,\infty),\Rset)}$.
	\end{enumerate}
	
	Then for every ${t_0 \in (a,\infty)}$ we have that 
\begin{equation} \label {eq3}
{I^{\alpha}_a\,T^{\alpha}_a f(t_0) = f(t_0) - f(a)}.
\end{equation}
\end{prop}
\begin{proof}
Condition 1 states that $f$ is right continuous at $a$, i.e.  $\lim\limits_{t \to a+} f(t)=f(a)$ and
then the statement follows from Proposition \ref{p2.11}.
\end{proof}

\begin{rem} \label{r2.10}
Note that the condition 1. in Proposition \ref{p2.9} is essential and can not be replaced with changing the condition 2. to  
${f \in C_a^{\alpha}([a,\infty),\Rset)}$.
The reason is the Definition \ref{d2.2} where $T_a^\alpha f(a)$ does not depend from the value of $f(a)$ 
and allows for example $f$ to have a jump in the low terminal point $a$. 
Hence without the condition 1. any involvement of $f(a)$ in the right side of \eqref{eq3} will be a mistake.
The problem is well illustrated with the simple example below. 
\end{rem}


\begin{exmp}
Let $\alpha \in (0,1)$ and  for $f:[a,\infty) \to \Rset$ be fulfilled ${f \in C_a^{\alpha}([a,\infty),\Rset)}$ 
and $f$ be right continuous at $a$.
Let define $g:[a,\infty) \to \Rset$ as follows:
\begin{equation} \nonumber
g(t) =
\begin{cases}
& f(t),  \ \ \ \ \ \ \ \text{for} \ t \in (a,\infty), \\
& f(t) + c, \ \ \text{for} \ t=a,
\end{cases}
\end{equation}
where $c \in \Rset$ and $c \neq 0$.

Since ${f \in C_a^{\alpha}([a,\infty),\Rset)}$ and taking in account Definition \ref{d2.2} we can conclude, 
that $T_a^\alpha g(t) = T_a^\alpha f(t)$ for every $t \in [a,\infty)$.
Hence for every $t \in (a,\infty)$ it holds  also $I^{\alpha}_a\,T^{\alpha}_a g(t) =I^{\alpha}_a\,T^{\alpha}_a f(t)$,
but obviously $ g(t) - g(a) = f(t) - f(a) - c \neq f(t) - f(a)$.

That is why for the function $g$ Proposition \ref{p2.11} is applicable, but Proposition \ref{p2.9} not. 
\end{exmp}


Now let see what happens in the lower terminal of the conformable derivative.

\begin{thm}   [\cite{KPZ19}]   \label{t2.12}
	Let ${f : [a,\infty) \to \Rset}$ and there exists a number ${\alpha \in (0, 1]}$ 
	such that the conformable derivative ${T^{\alpha}_a f(a)}$ with lower terminal point $a$ exists.
	
	Then the conformable derivative ${T^{\beta}_a f(a)}$ also exists for every ${\beta \in (0,\alpha)}$ and 
$$
	T^{\beta}_a f(a) = 0.
$$
\end{thm}

It is obviously that the following corollary of the above Theorem \ref{t2.12} holds:

\begin{cor}    \label{c2.13}
	Let ${f : [a,\infty) \to \Rset}$ and there exists a number ${\alpha \in (0, 1]}$ 
	such that the conformable derivative ${T^{\alpha}_a f(a)}$ with lower terminal point $a$ exists
and $T^{\alpha}_a f(a) \neq 0$.
	
	Then for every ${\beta \in (\alpha,1]}$  the conformable derivative ${T^{\beta}_a f(a)}$ does not exist.
\end{cor}

Now we can summarize for the standard (in our opinion) case,
when for some $\alpha \in (0,1)$ the conformable derivative 
${T^{\alpha}_a f(a)}$ exists and ${T^{\alpha}_a f(a)=c\neq 0}$:

\begin{equation} \nonumber
T_a^\beta f(a) =
\begin{cases}
& \text{does not exist},                  \ \text{for} \ \beta \in (\alpha,1]; \\
& c, \ \ \ \ \ \ \ \ \ \ \ \ \ \ \ \ \ \ \text{for} \ \beta=\alpha; \\
& 0, \ \ \ \ \ \ \ \ \ \ \ \ \ \ \ \ \ \ \text{for} \ \beta \in (0,\alpha).
\end{cases}
\end{equation}

And for $f'(a)$ we have no information, because in the general case $f'(a)$ does not match with $T_a^1 f(a)$.

\begin{rem}
This situation is not very comfortable, 
because by the most real application of the fractional derivatives 
it is usual  first a fractional model of worldwide phenomena to be constructed 
and then by changing the fractional order the best approximation to the empirical data to be sought.
\end{rem}


\section{Problem Statement} \label{sec:3}

And here we have the main problem - 
how to redefine the conformable derivative $T_a^\alpha$ 
in the lower terminal $a$ (i.e. to replace Definition \ref{d2.2} with another one) so that:
\begin{enumerate}
\item [$1.$] The definition to look natural - 
ie. to be analogous to the definition of the first order derivative 
as well as to the definition of the derivative in the lower terminal for any of the classical fractional derivatives
(in the case of Caputo derivative it is $0$).
\item [$2.$] ${T_a^\alpha f(a)}$ to depend from $f(a)$ 
and from $\alpha$-differentiability at $a$ to follow right continuity of $f$ at $a$.
\item [$3.$] ${T_a^\alpha f(a)}$ and ${T_a^\beta f(a)}$ to may or may not exist simultaneously
for every $\alpha , \beta \in (0,1]$.
\item [$4.$] ${T_a^\alpha f(a)}$ and $f'(a)$ to may or may not exist simultaneously
for every $\alpha \in (0,1]$.
\item [$5.$] To be fulfilled  $T_a^1 f(a) =f'(a)$.
\item [$6.$] The formulas in Theorem \ref{t2.6} and Corollary \ref{c2.7} 
for a direct connection between the conformable derivatives of different order to be fulfilled for the lower terminal too.
\end{enumerate}

It is clear, that with Definition \ref{d2.2} no one of the listed points above is fulfilled.


\section{A Solution of the Problem} \label{sec:4}

As replacement of Definition  \ref{d2.2} we offer the following definition:

\begin{defn}   \label{d4.1}
The conformable ${\alpha}$-derivative of $f$ at the lower terminal point $a$ 
exists if and only if the right first derivative $f'(a)$ in $a$ exists 
and is defined as 
\begin{equation} \nonumber
T_a^\alpha f(a) =
\begin{cases}
& f'(a),  \  \ \text{for} \ \alpha=1; \\
& 0, \ \ \ \ \ \ \ \text{for} \ \alpha \in (0,1).
\end{cases}
\end{equation}
\end{defn}  

Obviously all desired properties, described in the previous Section \ref{sec:3} will be fulfilled
if Definition \ref{d2.2} is replaced with Definition \ref{d4.1}.


\section{Some Consequences} \label{sec:5}

Let the conformable derivative be defined with Definition \ref{d2.1} and Definition \ref{d4.1}.

Below we list 
how some of the statements, presented in the Section Preliminaries,
will change with the replacement of the definition of the conformable derivative.

In Theorem \ref{t2.4}, Theorem \ref{t2.5}, Theorem \ref{t2.6} and Corollary \ref{t2.5} 
the lower terminal can be included in the range as shown below:

\begin{thm}  \label{t5.1}
Let ${\alpha \in (0,1]}$.
If $f:[a,\infty) \to \Rset$  is $\alpha$-differentiable at $t_0 \in [a,\infty)$,
then $f$ is continuous at $t_0$. 
\end{thm}
\begin{proof}
For $t_0 \in (a,\infty)$ the proof will be the same as the proof of Theorem \ref{t2.4} and will be omitted.

Let $t_0=a$. Since $f$  is $\alpha$-differentiable at $a$, from Definition \ref{d4.1} it follows that 
$f$ has right first derivative at $a$ and hence it is right continuous at $a$.  
\end{proof}

\begin{thm}  \label{t5.2}
	Let ${\alpha \in (0,1]}$, $c,d \in \Rset$ and $J \subset [a,\infty)$. 
We assume that ${f,g \in C^{\alpha}_a (J, \Rset)}$. 
	Then for ${t \in J}$ the following relations hold:
	\begin{enumerate}
		\item [$(i)$] ${T^{\alpha}_a (cf+dg) = c\,T^{\alpha}_a f + d\,T^{\alpha}_a g}$;\vspace{0.1cm}
		\item [$(ii)$] ${T^{\alpha}_a (fg) = g\,T^{\alpha}_a f + f\,T^{\alpha}_a g}$;\vspace{0.1cm}
		\item [$(iii)$] ${T^{\alpha}_a (fg^{-1}) = \left(f\,T^{\alpha}_a g - g\,T^{\alpha}_a f\right)g^{-2}}$;\vspace{0.1cm}
		\item [$(iv)$] ${T^{\alpha}_a (1) = 0}$;\vspace{0.1cm}
		\item [$(v)$] ${T^{\alpha}_a f(t) = (t - a)^{1-\alpha}f'(t)}$ 
(excluding the case of $T_a^1 f(a)$).
	\end{enumerate}
\end{thm}
\begin{proof}
For statements $(i) - (iv)$: \\
The proof for $t \in J, t\neq a$ will be the same as the proof of Theorem \ref{t2.5} and will be omitted.
For the case when $t=a$ and $\alpha=1$ the statements are fulfilled, because they will deal with first derivatives.
For the case when $t=a$ and $\alpha \in (0,1)$ the statements are obviously fulfilled, 
because the $\alpha$-derivatives at $a$ will be equal to zero.

For statement $(v)$ see the proof of Corollary \ref{c5.4} 
\end{proof}

\begin{thm}    \label{t5.3}
	Let ${f : [a,\infty) \to \Rset}$ and there exist a point ${t_0 \in [a,\infty)}$ and number ${\alpha \in (0, 1)}$ 
	such that the conformable derivative ${T^{\alpha}_a f(t_0)}$ with lower terminal point $a$ exists.
	
	Then the conformable derivative ${T^{\beta}_a f(t_0)}$ exists 
	for every ${\beta \in (0, 1]}$ with ${\beta \neq \alpha}$ and 
$$
	{T^{\alpha}_a f(t_0) = (t_0 - a)^{\beta-\alpha}T^{\beta}_a f(t_0)},
$$
or for the case when $t_0=a$ and $\beta < \alpha$
$$
	{T^{\beta}_a f(t_0) = (t_0 - a)^{\alpha-\beta}T^{\alpha}_a f(t_0)}.
$$

\end{thm}
\begin{proof}
For $t_0 \in (a,\infty)$ the proof will be the same as the proof of Theorem \ref{t2.6} and will be omitted.
For $t_0=a$ the statement follows directly  from Definition \ref{d4.1}.  
\end{proof}

\begin{cor}    \label{c5.4}
	For a function ${f : [a,\infty) \to \Rset}$ the conformable derivative ${T^{\alpha}_a f(t_0)}$ 
	with lower terminal $a$ at a point ${t_0 \in [a,\infty)}$ for some ${\alpha \in (0,1)}$ exists 
	if and only if 
	the function ${f(t)}$ has first derivative at the point ${t_0 \in [a,\infty)}$ and
\begin{equation} \label{ee1}
	{T^{\alpha}_a f(t_0) = (t_0 - a)^{1-\alpha}  f'(t_0)}.
\end{equation}
\end{cor}
\begin{proof}
For $t_0 \in (a,\infty)$ the proof will be the same as the proof of Corollary \ref{c2.7} and will be omitted.
For $t_0=a$ the statement follows from Definition \ref{d4.1}.  
\end{proof}

The main advantage of the changed definition can be seen in the new theorem below, 
which replaces both Proposition \ref{p2.9} and  Proposition \ref{p2.11}.

\begin{thm}    \label{t5.5}
	Let the  for some ${\alpha \in (0,1)}$ the function ${f \in C_a^{\alpha}([a,\infty),\Rset)}$.
		
	Then for every ${t_0 \in (a,\infty)}$ we have that 
\begin{equation} \label {eq4}
{I^{\alpha}_a\,T^{\alpha}_a f(t_0) = f(t_0) - f(a)}.
\end{equation}
\end{thm}

\begin{proof}
Since  ${f \in C_a^{\alpha}([a,\infty),\Rset)}$ then $f$ is $\alpha$-differentiable at $a$. 
Then according Definition \ref{d4.1} $f$ is right differentiable at $a$ and hence right continuous at $a$. 
The rest of the proof is the same as the proof of Proposition \ref{p2.9}.
\end{proof}

%
%

\section{Conclusions}
\label{sec:Conclusions}

In this paper we point out and discuss some unusual consequences 
raised by the definition of the conformable derivative in the lower terminal. 

A replacement for this definition is proposed and new rewritten statements adjusted to the new definition are presented.

Yes, it is about changing the behaviour of the conformable derivative at a single point, 
but it is an important one. 
It is sufficient to mention that for example exactly this point is usually used to construct initial value problems 
for any kind of differential equations. 

We are convinced that the definition of the conformable derivative thus corrected 
is more convenient and natural than the one currently existing and would be well to be used in the future.

%

%
%

\end{document}